\documentclass{amsart}

\usepackage{graphics} 
\usepackage{subcaption} 

\newcommand\eqsubref[1]{(\subref{#1})}
\usepackage{amsmath}
\usepackage{adjustbox}
\usepackage{amssymb}

\usepackage{tikz, tikz-cd}
\usepackage{fullpage}
\usepackage[all]{xy}
\usepackage[margin=1in]{geometry}
\usepackage{amsthm}
\usepackage{amsfonts}
\usepackage{bbm}
\usepackage{comment}
\usepackage{amsmath}
\usepackage{amsthm,diagbox}
\usepackage{bbm}
\usepackage{array} 
\usepackage{color}
\usepackage[utf8]{inputenc}
\usepackage[english]{babel}
\usepackage{hyperref}

\newcommand{\Z}{{\mathbb{Z}}}

\newcommand{\inj}{\hookrightarrow}

\newtheorem{theorem}{Theorem}[section]
\newtheorem{lemma}[theorem]{Lemma}

\theoremstyle{definition}  
\newtheorem{definition} [theorem] {Definition}

\newtheorem{question}{Question}

\newtheorem{thm}[theorem]{Theorem}

\theoremstyle{definition}
\newtheorem{rem}[theorem]{Remark}

\title{Studying links via book links: A Markov Theorem}

\author{Rom\'an Aranda}
\address{Department of Mathematics, University of Nebraska-Lincoln}
\email{jarandacuevas2@unl.edu}
\author{Fraser Binns}
\address{Department of Mathematics, Princeton University}
\email{fb1673@princeton.edu}
\author{Margaret Doig}
\address{Department of Mathematics, Creighton University}
\email{margaretdoig@creighton.edu}
\thanks{FB was supported by the Simons Grant {\em New structures in low-dimensional topology}}
\date{\today}

\begin{document}

\begin{abstract}
    We introduce ``book links" as a generalization of braids in open book decompositions; this new class of objects includes both braids and plats as special cases. We then prove a version of Markov's theorem in this general setting by extending the theory of open book foliations.
\end{abstract}

\maketitle

\section{Introduction}

In this paper, we introduce a concept that interpolates between the classical ideas of braids and plats. Braids are of wide mathematical interest in fields including topology, group theory, contact geometry, and algebraic geometry --- see the survey article~\cite{birman2005braids}. In particular, every link in $S^3$ can be encoded as a braid and thus understood via Artin's braid group~\cite{artin1925theorie}. Our proposed generalization of braids allows us, in principle, to reduce various topological questions from link theory to algebraic questions as in the braid case. It also allows us to extend the rich library of foliation-related techniques developed for braids to a larger class of links.

\subsection{Book links}

Let $Y$ be an arbitrary 3-manifold with a fixed open book decomposition, that is, a \emph{link} $B\subset Y$ with a fibration $p: (Y-B)\rightarrow S^1$ with $D_\theta$ pages, equivalently, a circle-valued Morse function on $Y-B$ without critical points. A link $L$ is an embedding $\bigsqcup S^1\inj Y-B$. The traditional definition of a \emph{closed braid} is a link that intersects each page $D_\theta$ in the fibration transversely, equivalently, a link where $p\circ L$ is also a Morse function without critical points. We define a \emph{closed book link} to be a link that intersects pages transversely except at finitely many non-degenerate critical points, equivalently, where $p\circ L$ is a Morse function. We will study book links up to \emph{book link isotopy} --- isotopies through book links, equivalently, through book links with the same number of critical points.  We will generally suppress references to the 3-manifold and open book decomposition and will assume that any link is embedded in a fixed 3-manifold equipped with an open book decomposition and that any book link structure is the one induced by that open book decomposition. One can define a book link analogue of an open braid as an appropriate non-degenerate tangle in a thickened surface, modulo isotopies rel boundary that fix the number of critical points.

\begin{figure}[hb]\centering
    \begin{subfigure}{0.2\textwidth}\centering
        \includegraphics[scale=0.5]{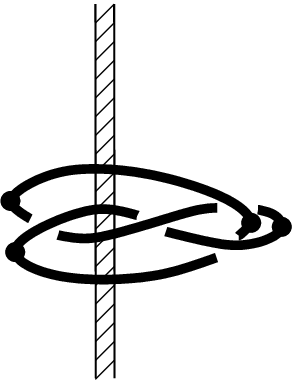}\caption{}\label{fig:fig8-20}
    \end{subfigure}
    \begin{subfigure}{0.2\textwidth}\centering
        \includegraphics[scale=0.5]{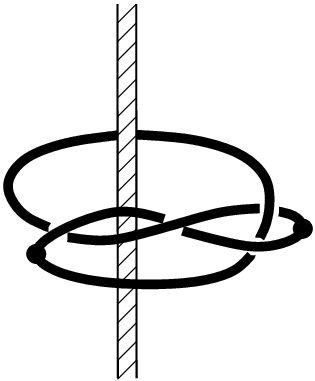}\caption{}\label{fig:fig8-11}
    \end{subfigure}
    \begin{subfigure}{0.2\textwidth}\centering
        \includegraphics[scale=0.5]{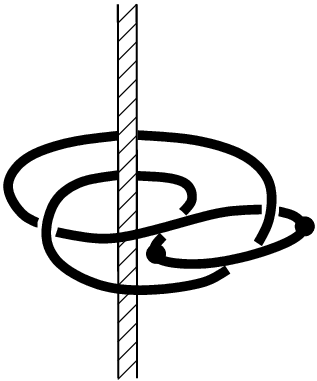}\caption{}\label{fig:fig8-12}
    \end{subfigure}
    \begin{subfigure}{0.2\textwidth}\centering
        \includegraphics[scale=0.5]{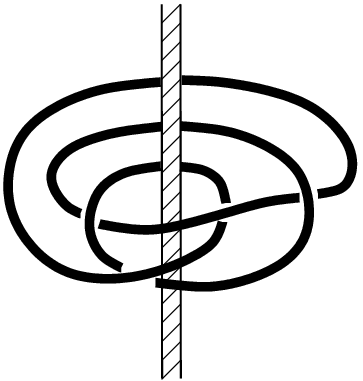}\caption{}\label{fig:fig8-03}
    \end{subfigure}
\caption{Book link realizations of the figure 8 knot which are \eqsubref{fig:fig8-20} 2-bridge, 0-braid, \eqsubref{fig:fig8-11} 1-bridge, 1-braid, \eqsubref{fig:fig8-12} 1-bridge, 2-braid, and \eqsubref{fig:fig8-03} 0-bridge, 3-braid. The solid dots indicate critical points while the binding is the vertical arc.}\label{fig:generic_knot}
\end{figure}

Each of these book link isotopy classes has a pair of invariants, the \emph{bridge index} or number of critical points, and the \emph{braid index} or minimum geometric intersection number with a page (minimized over the book link isotopy class). Just as a braid is a book link with bridge index zero, but potentially with high braid index, a plat closure can be viewed naturally as a book link that has arbitrarily high bridge index but braid index zero. The concept of a book link thus connects these two concepts. Figure~\ref{fig:generic_knot} shows four different book link realizations of the same knot in $S^3$ with unknotted binding and various bridge and braid indices. 

Throughout this paper the reader may find it helpful to take  $\mathbb{R}^3$ with cylindrical coordinates as a prototypical example of a manifold equipped with an open book. Here the $z$-axis is the binding of the open book and the pages $H_\theta$ are the points sharing the same angle $\theta$. In this case, a braided knot is an embedding {$\textbf{r}: S^1 \inj \mathbb{R}^3$} that misses the $z$-axis with tangent vector transverse to $H_\theta$, while a book knot is an embedding where, if the tangent vector is parallel to $H_\theta$, then the normal vector is not. Equivalently, if $\textbf{r} = (r,\theta,z)$ is parameterized using cylindrical coordinates, a braid obeys $\theta'(t) \neq 0$, while a book link has $\theta$ with only non-degenerate critical points, i.e., if $\theta'(t)=0$, then $\theta''(t)\neq 0$. The general case of a link in an open book decomposition of a 3-manifold is locally identical.

While book links are natural objects of study in their own right, they have appeared recently in the context of Floer homology.  For example, a key step in the proof of Baldwin-Sivek's classification of genus one ``nearly fibered" knots amounts to classifying bridge index one book links satisfying a simple topological condition~\cite[Section 6]{baldwin2022floer}. In forthcoming work, the second author and Dey classify links whose link Floer homology satisfies a simple algebraic condition; one case of their classification includes a family of bridge number one book links together with the bindings of the corresponding open books~\cite{HFalmostbraids}.

\subsection{Classification results}\label{section:mainq}
Given that book links are a topological generalization of braids, we have the following natural question: To what extent do topological results concerning braids extend to book links? In this paper, we address two fundamental instances of this question. 

One foundational result in the study of braids is \emph{Alexander's Theorem}, that any link in any three-manifold can be isotoped to a braid~\cite{alexander1923lemma,pavelescu2008braids}. The analogous statement in the setting of book links is discussed in Section~\ref{sec:Alex}. 

\begin{theorem}[Alexander's Theorem for book links]\label{thm:AlexV1}
    Let $Y$ be a 3-manifold equipped with an open book decomposition. Any link $L\subset Y$ can be isotoped to a book link $\lambda$ with any desired (even) number of critical points on each component. 
\end{theorem}

This result can also be thought of as a generalization of the fact that every link can be put bridge position. Another foundational result in braid theory is \emph{Markov's Theorem}, which characterizes when distinct braid representatives of a fixed link are isotopic~\cite{markov1935uber, birman1975braids}. The plat version of the Markov Theorem was proven by Birman in \cite{Birman1976MarkovForPlats}. Our main goal in this paper is to generalize Markov's theorem to the book link setting. We say two book links are of the same \emph{type} if they are of the same link type and the corresponding components have the same number of critical points. 

\begin{theorem}[Markov's Theorem for book links]\label{thm:Markov}
    Let $Y$ be a 3-manifold equipped with an open book decomposition. If $\lambda_1$ and $\lambda_2$ are book links in $Y$ of the same type, then they are related by a sequence of book link isotopies and stabilizations and destabilizations.
\end{theorem}

Here the operations of stabilizations and destabilizations are defined exactly as in the braid case. It follows directly from the theorem that any two book link representatives of the same link type --- even if not the same book link type --- can be connected via book link isotopies, stabilization/destabilization, and the creation/cancellation of critical points. 

To prove this theorem we extend the theory of braid foliations to surfaces with book link boundary. For braids, this machinery was first developed by Birman and Menasco in a sequence of papers~\cite{birman1992studying1,birman1991studying,birman1993studying,birman1990studying,birman1990studyingerratum,birman1992studying5,birman1992studying}. The text of LaFountain and Menasco gives an excellent exposition on braids and foliation techniques which we adapt here~\cite{lafountain2017braid}. 

\subsection{A spectrum of invariants}\label{section:basics}

We discuss invariants of links derived from book links which encompass and unify the classical braid and bridge indices. Given that every link can be isotoped to a book link with any bridge index (Theorem~\ref{thm:isotoping}), we can extract a spectrum of invariants from book links. More details appear in the paper of Doig and Gehringer~\cite{doigXXXXspectrum}, which shows the behavior of the spectrum for split and composite links and calculates the spectrum for prime knots through 9 crossings.

\begin{definition}
     The \emph{book index spectrum} of a link $L$ is the sequence $\big\{b_L(d)\big\}_{d=0}^\infty$, where $b_L(d)$ is the minimum braid index of a book link representative of $L$ with bridge index $d$.
\end{definition}

For a link $L$ in $S^3$ with the standard open book decomposition, the quantity $b_L(0)$ recovers the classical definition of the braid index of $L$, i.e., the minimum braid index over all braid representatives of the link. Similarly, the first $d$ for which $b_L(d)=0$ is exactly the classical bridge index of $L$. Observe that if $L$ has a book link representative of braid index $n$ and bridge index $d$ then $L$ admits a book link representative with braid index $n$ and bridge index $d+1$ that can be obtained by perturbing a segment of $L$ without critical points to introduce a pair of canceling critical points. It follows that $b_L:\Z^{\geq 0}\to \Z^{\geq 0}$ is a non-increasing function. 

One can show that the function $b_L$ is strictly decreasing between $0$ and the bridge index of $L$. It follows that torus links and other BB-links (those with the same classical bridge and braid index, say $n$) satisfy $b_L(d)=n-d$ for $0 \leq d \leq n$ and $b_L(d)=0$ else.  For an unknot $U$, $b_U(0)=1$ and $b_U(k)=0$ for $k> 0$. 

For a fixed link $L$, a book link representative of $L$ may be substantially ``simpler" than the simplest braid representative of $L$ in the sense that $b_L(d)$ may be arbitrarily larger than $b_L(d+1)$. For example, Schubert's classification of 2-bridge knots includes knots with arbitrarily high braid index~\cite{schubert1956knoten}, i.e., with $b_{L_n}(0)=n$, $b_{L_n}(1)=1$, and $b_{L_n}(i)=0$ for $i\geq 2$. By adding split unlinked components to $L_n$, it can be shown that there are links $L$ for which $b_L(d)-b_L(d+1)$ is arbitrarily large for any fixed $d$. 

The book index spectrum resembles the bridge spectrum of a knot introduced by Zupan to capture bridge indices of knots with respect to higher genus Heegaard surfaces of $S^3$ \cite{zupan2014bridge}. 

\subsection{Questions}\label{section:questions}
We end this section with some questions. First, recall that braids admit equivalent characterizations via mapping class groups or configuration spaces \cite{artin1925theorie,birman1975braids}. 

\begin{question}
    Do book links admit interpretations analogous to the mapping class or configuration space interpretations of braids?
\end{question}

\noindent Book links can be thought of as ``closures'' of tangles in thickened surfaces with the same number of critical points of index 0 and 1. With this in mind, one might relate such tangles to objects of a more algebraic flavour by appealing to an appropriately modified ``tangle category" --- see~\cite[Section~3]{turaev1990operator}. For example, in~\cite[Theorem 3.2]{turaev1990operator}, one could remove the relations corresponding to (a),(b), (c) and (f) in~\cite[Figure 4]{turaev1990operator}.  To move from an algebraic classification of such tangles to an algebraic classification of book link closures would require a slightly more subtle analysis than in the braid case, to account for isotopies which push the critical points of the book link through, for instance, every page of the open book.

Next is the problem of distinguishing book links up to book link isotopy. Given a book link $\lambda$ one can form the link $\lambda\cup B$ where $B$ is the axis of $\lambda$. Any classical link invariant of $\lambda\cup B$ is a book link isotopy invariant of $\lambda$. This implies, for instance, that the annular Whitehead link and its mirror --- $T_0$ and $T_{-1}$ --- can be distinguished via the signature of $T_i\cup B$.

\begin{question}
    What are the book link isotopy classes of the unknot with fixed braid and bridge indices?
\end{question}

\noindent Baldwin-Sivek implicitly gave a partial result in this direction, classifying specific types of braid index $3$ representatives of the unknot~\cite[Proof of Theorem 6.1]{baldwin2022floer}. They use this to prove various knot detection results for knot Floer homology and Khovanov homology. One might hope to obtain further link detection results by giving similar book link classifications.

We may also ask how difficult it is to distinguish book link isotopy classes:  
\begin{question}
    What is the computational complexity of determining whether two book links are book link isotopic?
\end{question}
\noindent For example, Hass, Lagaris, and Pippenger demonstrated the first result along these lines by showing that distinguishing a PL link from the unlink is in NP, as is identifying a link as split~\cite{hass1999computational}.

Open book foliations have a number of important topological applications. 
\begin{question}
    Can the techniques of open book foliations be extended to study surfaces in book link complements?
\end{question}
\noindent For example, open book foliations can be used to prove the Bennequin-Eliashberg inequality~\cite[Theorem~4.3]{ito2014open} and the Jones conjecture~\cite{lafountain2015embedded}, to define practical invariants such as self-linking number for transverse knots, and to relate the geometry of certain $3$-manifolds to the fractional Dehn twist coefficients of their possible open book structures~\cite[Theorem~8.3]{ito2012essential}. We note that there has also been some recent work adapting aspects of braid foliation theory to the setting of plat closures~\cite{menasco2024studying, solanki2023studying}, which suggests possible success in extending these techniques to the overarching book link case.

Similarly, it may be enlightening to extend other braid techniques to book links.
\begin{question}
    Do any braid invariants --- for example the self-linking number --- have book link analogues?
\end{question}

Finally, given Theorem~\ref{thm:Markov}, we have the following natural question:
\begin{question}
    Fix $k\geq 0$. For each $n\in\Z$, do there exist braid index $n$ book links which are only booklink isotopic after exactly $k$ stabilizations?
\end{question}

\subsection*{Outline} In Section~\ref{sec:Alex}, we prove a stronger version of Theorem~\ref{thm:AlexV1}, Alexander's Theorem for book links. In Section~\ref{section:foliations}, we consider an embedded surface bounded by a book link and study the foliation induced upon it by the open book decomposition of the manifold. In Section~\ref{section:markov}, we prove Theorem~\ref{thm:Markov} using the open book foliations of embedded annuli bounded by book links. 

\subsection*{Acknowledgements}
This project began at the AMS North Central Sectional Meeting at Creighton University in Fall 2023. We would like to thank the organizers of this meeting and its special sessions.

\section{An Alexander Theorem for Book Links}\label{sec:Alex}

The concept of a book link is generic enough to have a version of Alexander's Theorem which allows us to convert any link into a book link with the desired number of critical points. Theorem \ref{thm:AlexV1} in the introduction is a particular case ($\beta =\emptyset$) of the following result. 

\begin{thm}[Alexander's Theorem for book links]\label{thm:isotoping}
    Let $Y$ be a 3-manifold with a fixed open book decomposition and $\beta$ be a book link. Then, for any $L\subset Y\setminus\beta$, the link $L\cup\beta$ can be isotoped to a book link $\lambda \cup \beta$, where $\lambda$ has any desired (even) number of critical points on each component, and the isotopy restricted to $\beta$ is a book link isotopy. 
\end{thm}

\begin{figure}
    \begin{subfigure}[c]{0.49\textwidth}\centering
        \begin{subfigure}{0.95\textwidth}\centering
            \includegraphics[scale=0.4]{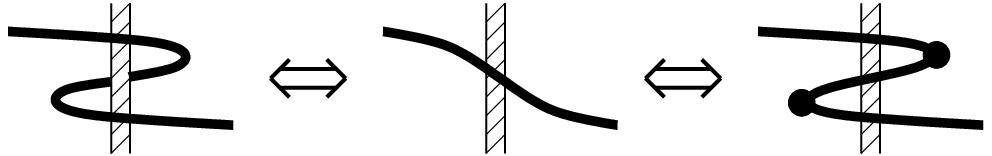}\caption{}\label{fig:add}
        \end{subfigure}
        \begin{subfigure}{0.95\textwidth}\centering
            \includegraphics[scale=0.4]{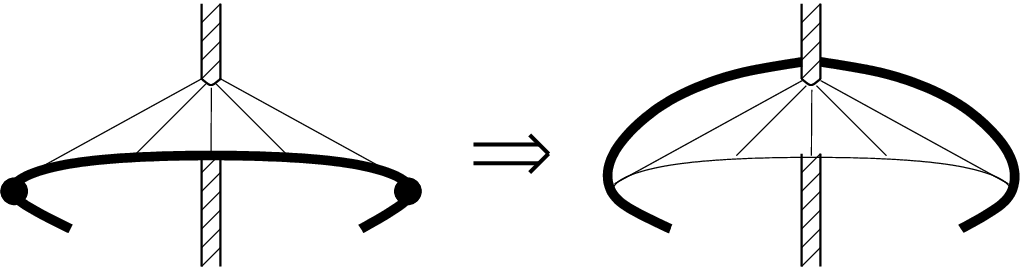}\caption{}\label{fig:cancel}
        \end{subfigure}
    \end{subfigure}
    \begin{subfigure}[c]{0.49\textwidth}\centering
        \includegraphics[scale=0.4]{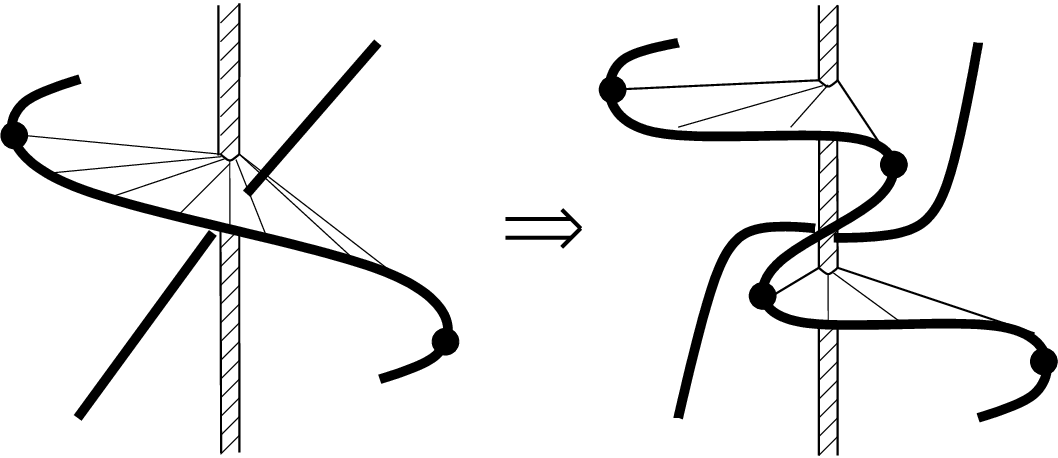}\caption{}\label{fig:cancelmulti}
    \end{subfigure}
    \caption{The book link structure and braid/bridge indices may be altered \eqsubref{fig:add} by stabilization and destabilization through the braid axis or by a local perturbation; and \eqsubref{fig:cancel} by an isotopy of a segment with a pair of critical points which has a free path to the braid axis. A segment without a free path may be modified as in \eqsubref{fig:cancelmulti}.}\label{fig:alexander}
\end{figure}
    
\begin{proof}
    A local perturbation will resolve any degenerate critical point into 1 or 2 non-degenerate critical points, i.e., we may immediately convert a link $L$ into a book link $\lambda$. 
    
    If $\lambda$ has too few critical points, it is easy to add additional ones, just as we may stabilize or destabilize in the traditional braid fashion; see Figure~\ref{fig:add}. This operation may be realized by an isotopy which is the identity away from the replacement area.
    
    If $\lambda$ has too many critical points, they may be resolved in pairs following an argument similar to that of Lemma~2.3 in \cite{lafountain2017braid}. First, suppose there are two critical points connected by a segment with no other critical points as in Figure~\ref{fig:cancel}. If there is a point on the braid axis and an interval family of continuously varying arcs, each lying in a single page, which connect the segment to that point, and if these paths are disjoint from the rest of $\lambda \cup \beta$, then they will sweep out a triangle in the complement of $\lambda \cup \beta$. We may isotope the segment through this triangle and push it to the other side of the braid axis; if we are careful, the intermediate segments will contain a matched pair of critical points until they converge in a single point when the segment intersects the braid axis. Then the intermediate segment passes through the braid axis and the critical points are eliminated.

    If it is impossible to find such an interval family of arcs that do not intersect another segment of $\lambda$ or $\beta$ as in Figure~\ref{fig:cancelmulti}, we may locally alter both segments to remove the issue. This will introduce a new pair of critical points, but the four points may now be pairwise resolved --- unless there is another intersecting segment to accommodate likewise, but this process will terminate in a finite number of steps. 
\end{proof}

We next consider how an isotopy between two links may be reflected as a sequence of book links connected by unions of annuli. This an extension of Proposition~2.8 of \cite{lafountain2017braid} to book links and follows immediately by applying to the proposition Theorem~\ref{thm:isotoping}, but we include a proof here for thoroughness.

\begin{lemma}\label{lem:annuli}
   Suppose that $\lambda_-$ and $\lambda_+$ are book link representatives of a link in an arbitrary 3-manifold. There is a finite collection of pairwise disjoint book links $\lambda_-=\lambda_0$, $\lambda_1$, \dots, $\lambda_m=\lambda_+$ such that for each $i=1,\dots m$, there is a disjoint union of annuli $A_i$ with $\partial_-A_i = \lambda_{i-1}$ and $\partial_+A_i = \lambda_i$. Furthermore, the braids $\lambda_1, \dots, \lambda_{m-1}$ can be chosen to be book links of the same type. 
\end{lemma}

\begin{proof}
Observe that there exist piecewise linear links $L_-$ and $L_+$ such that $\lambda_-$ and $\lambda_+$ cobound families of embedded annuli with $L_-$ and $L_+$ respectively.
Pick a triangulation of $Y$ so that $L_-$ and $L_+$ lie in the 1-skeleton. After a finite number of subdivisions, the $L_-$ and $L_+$ are related by a sequence of $\Delta$-moves. Here, a $\Delta$-move corresponds to taking an embedded triangle with one edge on $L$ and replacing the arc on $L$ with the two remaining edges of the triangle. We can subdivide the triangulation to ensure that no successive pairs of triangles intersect. Notice that each $\Delta$-move traces a region that can be widened to an embedded annulus. 

The sequence of $\Delta$-moves can be used to build a family of $n$-tuples of oriented embedded annuli $\{A_i\}_{i=1}^m$ overlapping at their boundaries in links denoted by $\lambda_i$, and, in fact, they do not intersect in their interiors, i.e., $\partial_-A_i = \lambda_{i-1}$ and $\partial_+A_i=\lambda_i$, as well as $A_i \cap A_{i+1} = \lambda_i$. Theorem~\ref{thm:isotoping} guarantees the existence of an isotopy that converts each intermediate $\lambda_i$ into a book link with the appropriate number of critical points. This isotopy can be extended to an ambient isotopy, and the images of the unions of annuli under this isotopy are the desired $A_i$.
\end{proof}

\begin{figure}\centering
    \begin{subfigure}{0.24\textwidth}\centering\includegraphics[scale=0.33]{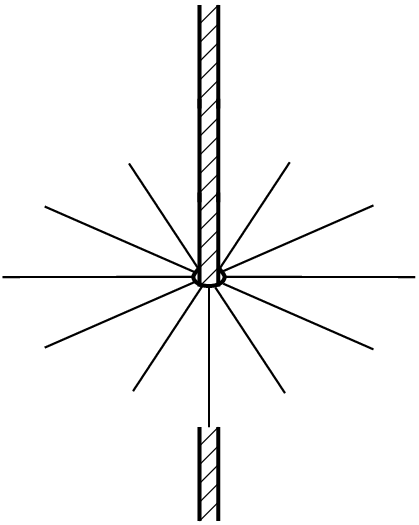}\caption{}\label{fig:ellipticpoint}\end{subfigure}
    \begin{subfigure}{0.24\textwidth}\centering\includegraphics[scale=0.33]{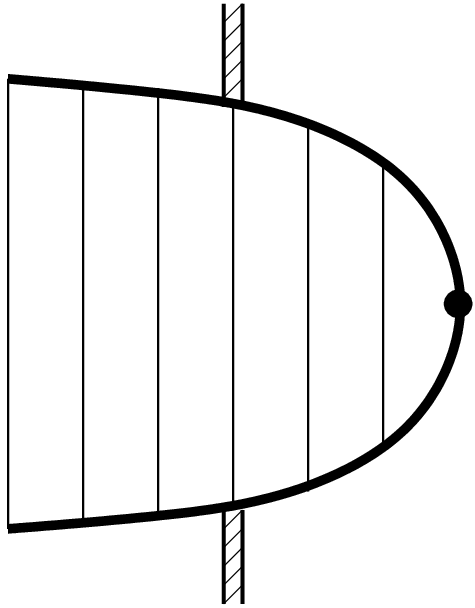}\caption{}\label{fig:boundarysingularity1}\end{subfigure}
    \begin{subfigure}{0.24\textwidth}\centering\includegraphics[scale=0.33]{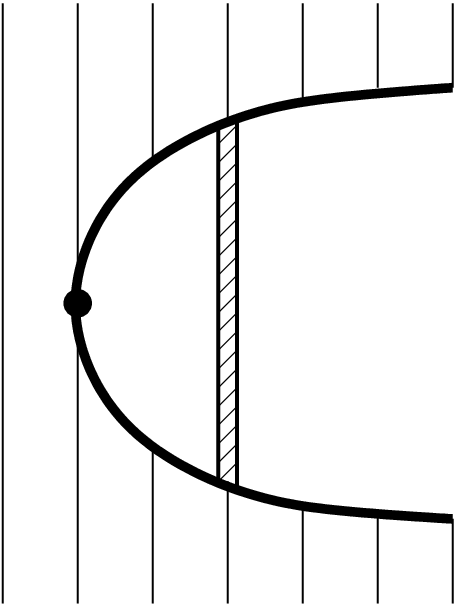}\caption{}\label{fig:boundarysingularity2}\end{subfigure}
    \begin{subfigure}{0.24\textwidth}\centering\includegraphics[scale=0.33]{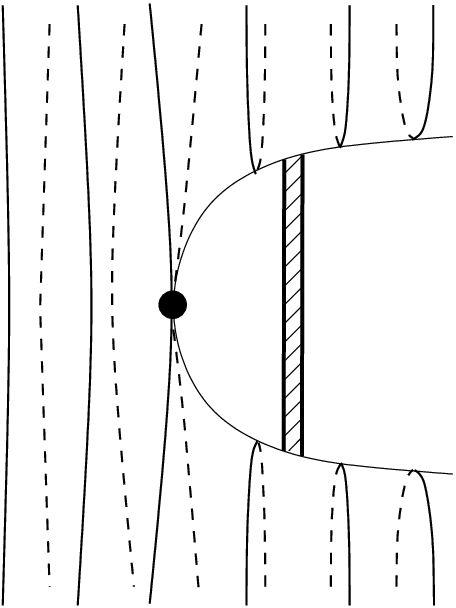}\caption{}\label{fig:saddlesingularity}\end{subfigure}
    \caption{\eqsubref{fig:ellipticpoint} An elliptic point where the binding meets the surface, (\subref{fig:boundarysingularity1}- \subref{fig:boundarysingularity2}) extremal singularities on the boundary, and \eqsubref{fig:saddlesingularity} a hyperbolic singularity in the interior of an open book foliation.} \label{fig:singularity3D}
\end{figure} 

\section{Foliations, Singularities, and Tiles}\label{section:foliations}

To understand surfaces bounded by book links, as well as to prepare for the proof of Theorem~\ref{thm:Markov}, we study foliations induced by the pages of the open book decomposition of the manifold. In particular, we investigate the resulting singularities and their foliated neighborhoods. 

We first expand Ito and Kawamuro's definition of an open book foliation~\cite{ito2014open} to accommodate surfaces bounded by book links and allow extremal singularities on the boundary only.

\begin{definition}
    Consider a surface $\Sigma$ in a 3-manifold $Y$ equipped with an open book decomposition $(B,p)$. Say $\partial\Sigma = \lambda$ is a book link and $\mathcal{F}(\Sigma)$ is the foliation on $\Sigma$ induced by the open book decomposition. We call $\mathcal{F}(\Sigma)$ an \emph{open book foliation} if the following conditions hold: $\Sigma$ is pierced by the binding in finitely many points, at each of which the foliation is radial; all but finitely many pages are transverse to $\Sigma$, and each exceptional page is tangent at a single point; and all tangencies are either extremal on $\partial \Sigma$ or of saddle type in $\text{int}\Sigma$. 
\end{definition}

We call a component of the intersection of the surface with a single page a \emph{leaf}, and it is either \emph{singular} or \emph{regular}. An \emph{elliptic point} is an intersection point of the surface and the binding, as in Figure~\ref{fig:ellipticpoint}. An \emph{extremal singularity} is a point where the boundary is tangent to a page (equivalently, a local maximum or minimum with respect to the Morse function), as in Figure~\ref{fig:boundarysingularity1} and~\ref{fig:boundarysingularity2}. A \emph{hyperbolic singularity} is a saddle point on the surface (equivalently, an index 1 critical point with respect to the Morse function), as in Figure~\ref{fig:saddlesingularity}. Note that elliptic and hyperbolic singularities may be given signs based on whether the orientation of the surface agrees with the orientation of the binding or the pages, respectively.

An open book foliation can have three types of leaf: arcs with both endpoints on $B$; arcs with an endpoint on $\partial\Sigma$ and an endpoint on $B$; arcs with both endpoints on $\partial\Sigma$.

\begin{lemma}\label{lem:generalposition}
    A surface with boundary a book link may be isotoped without changing the boundary to admit an open book foliation.
\end{lemma}

\begin{proof}
    We perturb the surface in such a manner that no two hyperbolic singularities occur on the same page. We may further assume that saddle points occur in the interior, or else we may move them into the interior by a small perturbation. 

    The work in \cite[Theorem~2.5]{ito2014open} allows the removal of simple closed curves and interior extremal critical points in the case of a surface bounded by a braid. The same arguments apply in our setting because they only require a modification of the interior of the surface and not its boundary.
\end{proof}

For the remainder of this paper, we fix an arbitrary 3-manifold, $Y$, equipped with an arbitrary open book decomposition $(B,p)$. We assume that any surface $\Sigma$ we refer to is embedded in $Y$ and equipped with the open book foliation induced by $(B,p)$.

Next, we examine how singularities of an open book foliation lie with respect to one another in the surface and decompose the surface into a set of tiles following the work of \cite[p.~30]{birman1992studying1} for braids in $S^3$ and \cite[\S2.1.3]{ito2014open} for braids in an arbitrary $Y^3$; see also the account in~\cite[Lemma~2.4]{lafountain2017braid}. Recall each arc contains at most one hyperbolic or extremal singularity because an arc comes from an intersection of a page with the surface, and we have assumed that these singularities occur on discrete pages. The singular leaves are separated from one another by product foliations, in other words, the surface can be divided along non-singular leaves into a finite set of regions which are neighborhoods of singular leaves. We will commonly visualize these regions by flattening the surface and drawing lines to represent the foliation, which we show in Figures~\ref{fig:tiles}. 

\begin{figure}\centering
    \begin{subfigure}{0.19\textwidth}\centering
        \includegraphics[scale=0.25]{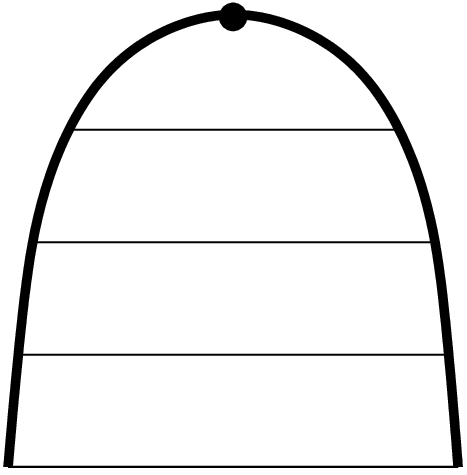}\caption{}\label{fig:boundarysingularity1flat}
        \end{subfigure}
    \begin{subfigure}{0.19\textwidth}\centering
        \includegraphics[scale=0.25]{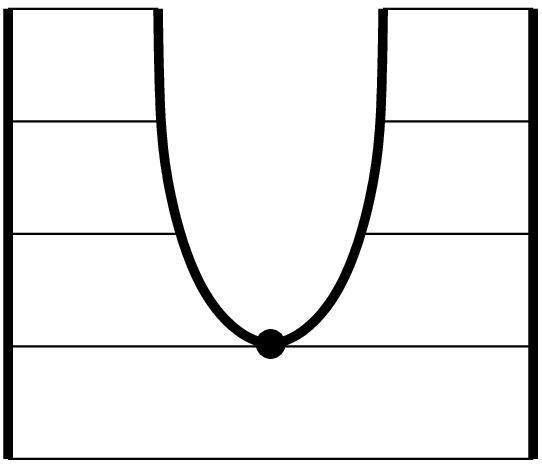}\caption{}\label{fig:boundarysingularity2a}
    \end{subfigure}
    \begin{subfigure}{0.19\textwidth}\centering
        \includegraphics[scale=0.25]{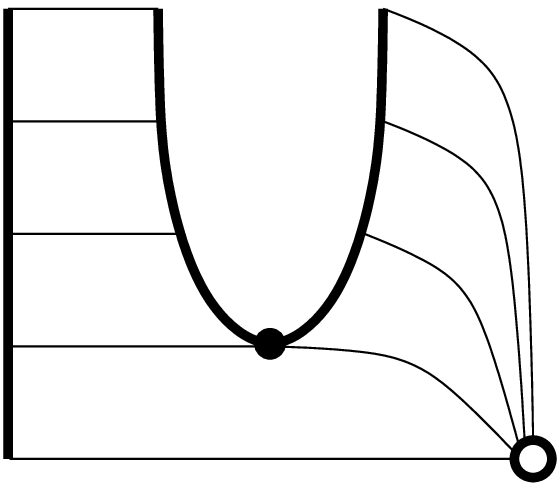}\caption{}\label{fig:boundarysingularity2b}
    \end{subfigure}
    \begin{subfigure}{0.19\textwidth}\centering
        \includegraphics[scale=0.25]{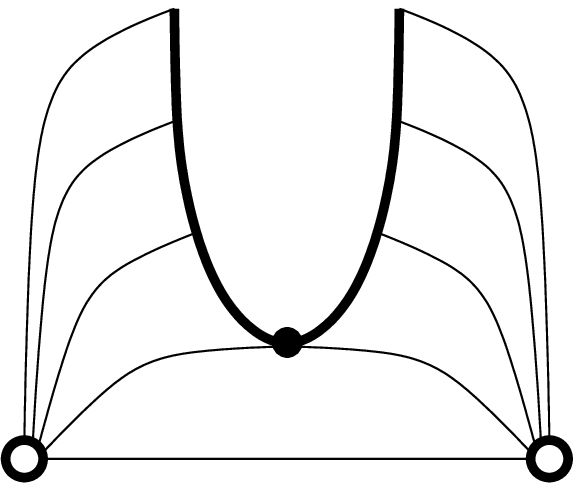}\caption{}\label{fig:boundarysingularity2c}
    \end{subfigure}
    \begin{subfigure}{0.19\textwidth}\centering
        \includegraphics[scale=0.25]{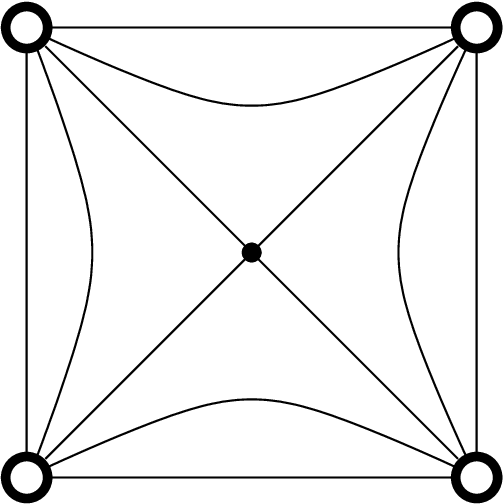}\caption{}\label{fig:interiorhyperbolics4}
    \end{subfigure}
    \par\medskip
    \begin{subfigure}{0.19\textwidth}\centering
        \includegraphics[scale=0.25]{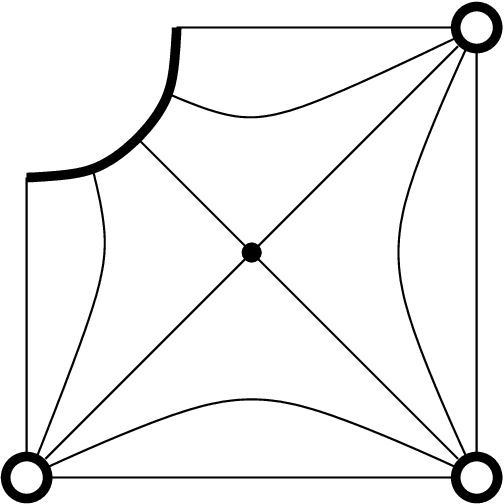}\caption{}\label{fig:interiorhyperbolics3}
    \end{subfigure}
    \begin{subfigure}{0.19\textwidth}\centering
        \includegraphics[scale=0.25]{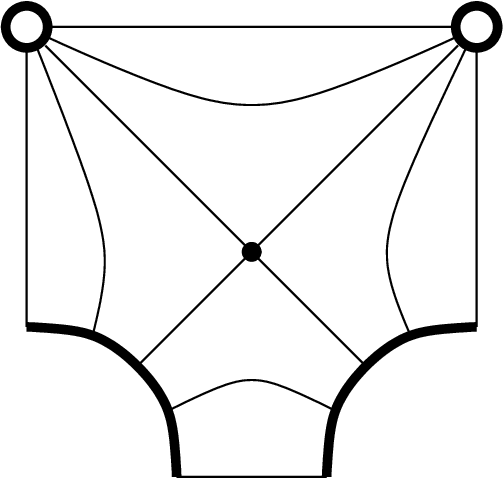}\caption{}\label{fig:interiorhyperbolics2}
    \end{subfigure}
    \begin{subfigure}{0.19\textwidth}\centering
        \includegraphics[scale=0.25]{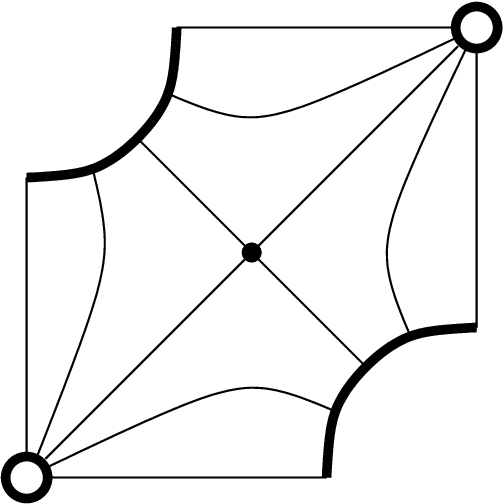}\caption{}\label{fig:interiorhyperbolics2var}
    \end{subfigure}
    \begin{subfigure}{0.19\textwidth}\centering
        \includegraphics[scale=0.25]{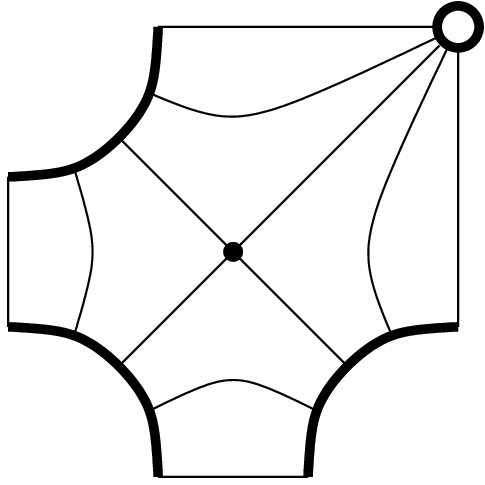}\caption{}\label{fig:interiorhyperbolics1}
    \end{subfigure}
    \begin{subfigure}{0.19\textwidth}\centering
        \includegraphics[scale=0.25]{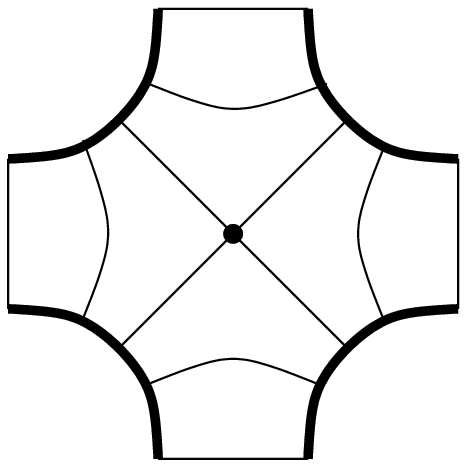}\caption{}\label{fig:interiorhyperbolics0}
    \end{subfigure}
    \caption{Tiles for boundary extremal and hyperbolic interior singularities (solid dots), which are connected by singular leaves to either boundary segments (dark lines) or elliptic points (open dots). Lemma~\ref{lem:onlytype0} explains how to eliminate tiles \eqsubref{fig:boundarysingularity2a}-\eqsubref{fig:boundarysingularity2c} by local perturbations and tiles \eqsubref{fig:interiorhyperbolics3}, \eqsubref{fig:interiorhyperbolics2}, \eqsubref{fig:interiorhyperbolics1} by stabilizations. Lemma~\ref{lem:rotate} classifies the surfaces tiled by only the remaining tiles: a surface without boundary by \eqsubref{fig:interiorhyperbolics4}; a surface bounded by a braid by \eqsubref{fig:interiorhyperbolics2var}; a surface bounded by a non-braid book link by \eqsubref{fig:interiorhyperbolics0} and \eqsubref{fig:boundarysingularity1flat} together (the $h$- and $\partial$-tiles).}\label{fig:tiles}
\end{figure}

\begin{definition}
    A \emph{tile} is the neighborhood of a singular leaf. The tile in Figure~\ref{fig:boundarysingularity1flat} is called a \emph{$\partial$-tile} and the tile in Figure~\ref{fig:interiorhyperbolics0} is called an \emph{$h$-tile}.
\end{definition}

\begin{lemma}\label{lem:onlytype0}
    A surface with an open book foliation is (possibly after book link isotopies and stabilizations or destabilizations) either closed; a disk; a surface with braided boundary; or a surface with book link boundary which may be decomposed along non-singular leaves into a finite set of $\partial$-tiles and $h$-tiles.
\end{lemma}

\begin{proof}
    The work in \cite[Lemma~2.4]{lafountain2017braid} ensures that the surface has a finite decomposition into tiles; however, the introduction of extremal singularities along the boundary allows additional tile types beyond those demonstrated by braids. 
    
    As in the case that the boundary is braided, we may assume each elliptic point is connected by an arc to at least one hyperbolic singularity, else there is a disk or sphere component. We may also assume that each hyperbolic singularity falls on a singular leaf whose four endpoints fall either at elliptic points or on the boundary since we may perturb the surface so that no two critical points occur in the same page. Therefore, the singularities of Figure~\ref{fig:singularity3D} give rise to two possible extremal tiles and six possible hyperbolic tiles, whether the singular leaves end at elliptic points or on the boundary. See Figure~\ref{fig:tiles}. Note that there are strictly more possibilities here than in the case the boundary of the surface is assumed to be braided.
    
    Figure~\ref{fig:singularitytransforms} shows us how to remove some of these tiles. We may first perform a local perturbation on the interior of the surface to convert one type of boundary singularity into the other by adding a hyperbolic singularity, and so we may remove any instance of \eqref{fig:boundarysingularity2a}, \eqref{fig:boundarysingularity2b}, or \eqref{fig:boundarysingularity2c}, i.e., any extremal singularity lies in a $\partial$-tile. We may also stabilize the book link to remove a hyperbolic singularity and an elliptic point together \`a la \cite[Figure~2.22]{lafountain2017braid}. Thus, we may assume there are no tiles \eqref{fig:interiorhyperbolics3}, \eqref{fig:interiorhyperbolics2}, and \eqref{fig:interiorhyperbolics1}. 
    
    Once we have removed these tiles, if there is any tile remaining with $4$ elliptic points as in \eqref{fig:interiorhyperbolics4}, then the entire surface consists of only these tiles and so is closed. If there is a tile \eqref{fig:interiorhyperbolics2var}, then, again, the entire surface must be tiled with the same tile, so the boundary will have no critical points and will be braided. Finally, $h$-tiles from \eqref{fig:interiorhyperbolics0} can be glued to one another and to $\partial$-tiles, giving a component of the surface with book link boundary.
\end{proof}

\begin{figure}\centering
    \begin{subfigure}{0.45\textwidth}\centering\includegraphics[scale=0.33]{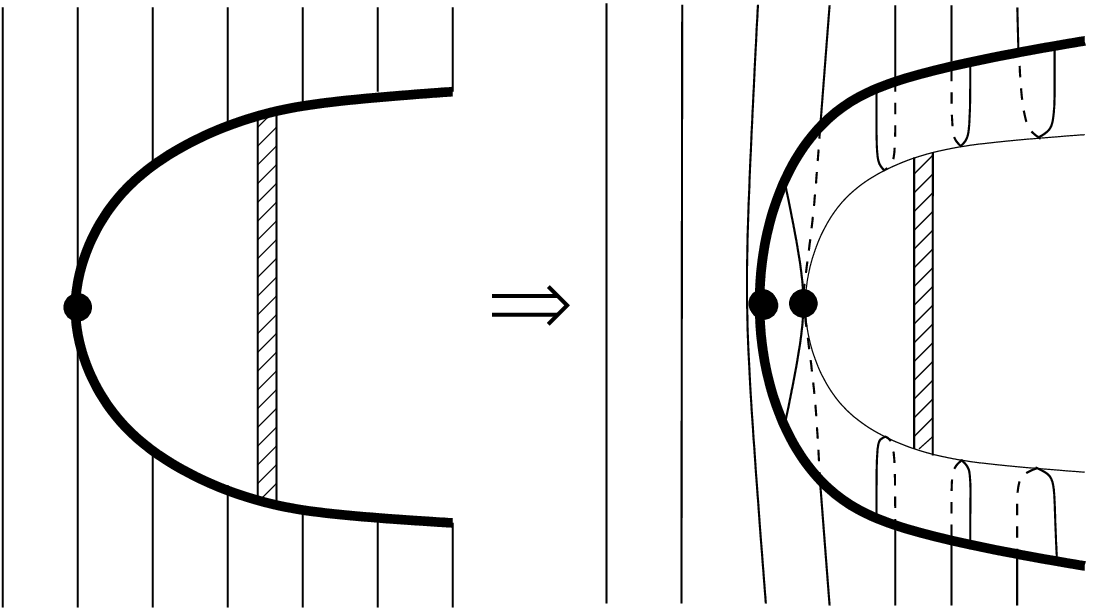}\caption{}\label{fig:boundarysingularitychange}\end{subfigure}
    \begin{subfigure}{0.45\textwidth}\centering\includegraphics[scale=0.33]{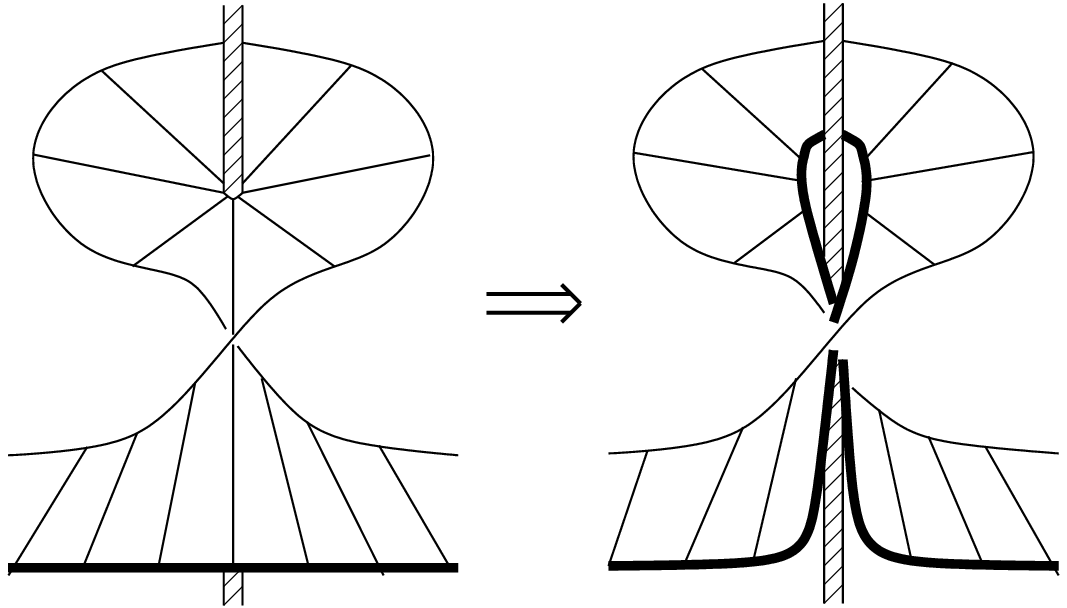}\caption{}\label{fig:singularitypair}\end{subfigure}
    \caption{\eqsubref{fig:boundarysingularitychange} A local perturbation eliminating one type of extremal singularity. \eqsubref{fig:singularitypair} A stabilization eliminating an elliptic point and a hyperbolic singularity.}\label{fig:singularitytransforms}
\end{figure}

We introduce one more tool to visualize the surface, in particular, to study how the singularities are connected to one another and the boundary. Represent a surface $\Sigma$ by a graph $\Gamma(\Sigma)$: assign a vertex for each singularity and connect two vertices with an edge if their tiles are adjacent. In the special case where $\Sigma$ has only $\partial$- and $h$-tiles, there will be a vertex of degree $4$ at each hyperbolic singularity and a vertex of degree $1$ at each extremal singularity. Additionally, if $\Sigma$ is an annulus with two boundary components, then the graph has a single simple cycle that divides the surface into two regions. If each boundary component is a book link of bridge index $d$, then each region has exactly $2d$ degree-1 vertices.

Finally, we will find it useful to perturb the surface locally around two hyperbolic singularities and examine how this perturbation is reflected in the transformation of the tiles and their graph. Similar techniques have previously been employed to study open book foliations, e.g., in~\cite[Figure~8]{birman1990studying} and~\cite[Theorem~3.1]{ito2014operations}.

\begin{lemma}\label{lem:rotate}
    Let $\Sigma$ be a surface with an open book foliation and two adjacent hyperbolic singularities. If the two hyperbolic singularities have the same sign, or if one is adjacent to a $\partial$-singularity, then they may be rearranged via an isotopy of $\Sigma$ as in Figure~\ref{fig:rotatetile} and its graph $\Gamma(\Sigma)$ altered as in Figure~\ref{fig:rotategraph}.
\end{lemma}

\begin{figure}
    \begin{subfigure}{0.95\textwidth}\centering
        \begin{subfigure}[c]{0.3\textwidth}\centering\includegraphics[scale=0.75]{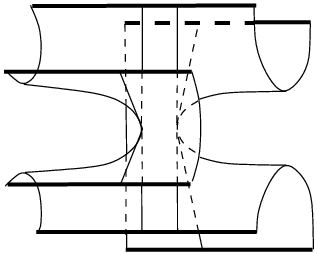}\end{subfigure}
        \begin{subfigure}[c]{0.35\textwidth}\centering\includegraphics[scale=0.75]{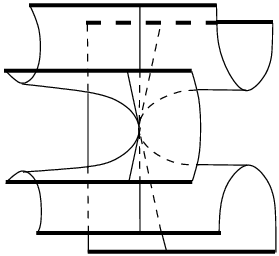}\end{subfigure}
        \begin{subfigure}[c]{0.3\textwidth}\centering\includegraphics[scale=0.75]{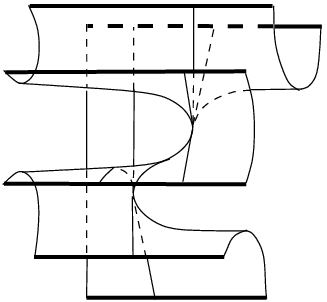}\end{subfigure}
        \caption{}\label{fig:rotate3d}
    \end{subfigure}
    \begin{subfigure}{0.95\textwidth}\centering
        \begin{subfigure}[c]{0.3\textwidth}\centering
            \input{rotate_before.tikz}
        \end{subfigure}
        \begin{subfigure}[c]{0.35\textwidth}\centering
            \input{rotate_mid.tikz}
        \end{subfigure}
        \begin{subfigure}[c]{0.3\textwidth}\centering
            \input{rotate_after.tikz}
        \end{subfigure}
        \caption{}\label{fig:rotatetile}
    \end{subfigure}
    \begin{subfigure}{0.95\textwidth}\centering
        \begin{subfigure}{0.3\textwidth}\centering
            \begin{tikzpicture}
                \node(H1) at (0,0) {$h_1$};
                \node(H2) at (1,0) {$h_2$} edge [->] (H1);
                \node(v1) at (-0.5,1) {$t_1$} edge [<-] (H1);
                \node(v2) at (0.5,1) {$t_2$} edge [->] (H2);
                \node(v3) at (1.5,1) {$t_3$} edge [<-] (H2);
                \node(v4) at (1.5,-1) {$t_4$} edge [->] (H2);
                \node(v5) at (0.5,-1) {$t_5$} edge [<-] (H1);
                \node(v6) at (-0.5,-1) {$t_6$} edge [->] (H1);
            \end{tikzpicture}
        \end{subfigure}
        \begin{subfigure}{0.35\textwidth}\centering
            \begin{tikzpicture}
                \node(H) at (0,0) {$h$};
                \node(v1) at (-0.8,1) {$t_1$} edge [<-] (H);
                \node(v2) at (0,1) {$t_2$} edge [->] (H);
                \node(v3) at (0.8,1) {$t_3$} edge [<-] (H);
                \node(v4) at (0.8,-1) {$t_4$} edge [->] (H);
                \node(v5) at (0,-1) {$t_5$} edge [<-] (H);
                \node(v6) at (-0.8,-1) {$t_6$} edge [->] (H);
            \end{tikzpicture}
        \end{subfigure}
        \begin{subfigure}{0.3\textwidth}\centering
            \begin{tikzpicture}
                \node(H1) at (0,0) {$h_1'$};
                \node(H2) at (1,0) {$h_2'$} edge  (H1);
                \node(v1) at (-0.5,1) {$t_1$} edge [<-] (H1);
                \node(v2) at (0.5,1) {$t_2$} edge [->] (H1);
                \node(v3) at (1.5,1) {$t_3$} edge [<-] (H2);
                \node(v4) at (1.5,-1) {$t_4$} edge [->] (H2);
                \node(v5) at (0.5,-1) {$t_5$} edge [<-] (H2);
                \node(v6) at (-0.5,-1) {$t_6$} edge [->] (H1);
            \end{tikzpicture}
        \end{subfigure}
        \caption{}\label{fig:rotategraph}
    \end{subfigure}
    \caption{Hyperbolic singularities with the same sign merging into a degenerate monkey saddle and then resolving: \eqsubref{fig:rotate3d} a foliated view, \eqsubref{fig:rotatetile} a tile view, and \eqsubref{fig:rotategraph} a graphical representation. The operation is shown with $h$-tiles for simplicity.}\label{fig:rotate}
\end{figure}

\begin{proof}
    If the two hyperbolic singularities have the same sign, i.e., they both agree (or both disagree) with the orientation of the pages, then Figure~\ref{fig:rotate3d} represents an ambient isotopy near these singularities which an isotopy through surfaces with fixed boundary and whose foliations are open book foliations except at the point where the surface has a single, degenerate critical point. Figures~\ref{fig:rotatetile} and~\ref{fig:rotategraph} represent an overhead, flattened view of the foliations and the graphs, respectively.

    If the two tiles have opposite signs but one is adjacent to a $\partial$-tile, recall the operation of Figure~\ref{fig:boundarysingularitychange}: the less desirable flavor of extremal boundary singularity may be resolved into a $\partial$-singularity and a hyperbolic singularity of either sign. Therefore, we may apply the move first in reverse and then forwards, selecting the opposite orientation, and then proceed with the isotopy in the figure.
\end{proof}

If desired, we may even ensure this operation does not change the number of hyperbolic singularities of each sign: after the operation, we may use a $\partial$-tile to reverse a sign.

\section{A Markov Theorem for Book Links}\label{section:markov}

Finally, we may apply the techniques of Section~\ref{section:foliations} to understand how to use collections of appropriate annuli book links to generate book link isotopies (after suitable stabilization).

\begin{figure}\centering
    \begin{subfigure}[c]{0.35\textwidth}\centering
        \begin{subfigure}{0.45\textwidth}
            \begin{center}\begin{tikzpicture}
                \node(H1) at (0,0) {$h$};
                \node(H2) at (1,0) {$h$} edge (H1);
                \node(t1) at (-0.5,1) {} edge[double] (H1);
                \node(t2) at (0.5,1) {} edge (H2);
                \node(t3) at (1.5,1) {} edge (H2);
                \node(t4) at (1.5,-1) {} edge (H2);
                \node(t5) at (0.5,-1) {} edge[double] (H1);
                \node(t6) at (-0.5,-1) {} edge (H1);
            \end{tikzpicture}\end{center}
        \end{subfigure}
        \begin{subfigure}{0.45\textwidth}
            \begin{tikzpicture}\centering
                \node(H1) at (0,0) {$h$};
                \node(H2) at (1,0) {$h$} edge[double] (H1);
                \node(t1) at (-0.5,1) {} edge[double] (H1);
                \node(t2) at (0.5,1) {} edge (H1);
                \node(t3) at (1.5,1) {} edge (H2);
                \node(t4) at (1.5,-1) {} edge (H2);
                \node(t5) at (0.5,-1) {} edge[double] (H2);
                \node(t6) at (-0.5,-1) {} edge (H1);
            \end{tikzpicture}
        \end{subfigure}
        \caption{}\label{fig:movehyperbolicgraph}
    \end{subfigure}
    \begin{subfigure}[c]{0.6\textwidth}\centering
        \begin{subfigure}[c]{0.65\textwidth}\centering
            \input{movehyperbolic_before.tikz}
        \end{subfigure}
        \begin{subfigure}[c]{0.3\textwidth}\centering
            \input{movehyperbolic_after.tikz}
        \end{subfigure}
        \caption{}\label{fig:movehyperbolictile}
    \end{subfigure}
    \begin{subfigure}{0.3\textwidth}\centering
        \begin{subfigure}{0.45\textwidth}\centering
            \begin{tikzpicture}
                \node(H1) at (0,0) {$h$};
                \node(H2) at (1,0) {$h$} edge [dashed] (H1);
                \node(v1) at (-0.5,1) {$h$} edge [dashed] (H1);
                \node(v2) at (0.5,1) {$\partial$} edge [dashed] (H2);
                \node(v3) at (1.5,1) {} edge (H2);
                \node(v4) at (1.5,-1) {} edge (H2);
                \node(v5) at (0.5,-1) {} edge (H1);
                \node(v6) at (-.5,-1) {} edge (H1);
            \end{tikzpicture}
        \end{subfigure}
        \begin{subfigure}{0.45\textwidth}\centering
            \begin{tikzpicture}
                \node(H1) at (0,0) {$h$};
                \node(H2) at (1,0) {$h$} edge (H1);
                \node(v1) at (-0.5,1) {$h$} edge [dashed] (H1);
                \node(v2) at (0.5,1) {$\partial$} edge [dashed] (H1);
                \node(v3) at (1.5,1) {} edge (H2);
                \node(v4) at (1.5,-1) {} edge (H2);
                \node(v5) at (0.5,-1) {} edge (H2);
                \node(v6) at (-0.5,-1) {} edge (H1);
            \end{tikzpicture}
        \end{subfigure}
        \caption{}\label{fig:movepartialgraph}
    \end{subfigure}
    \begin{subfigure}{0.34\textwidth}\centering
        \begin{subfigure}{0.45\textwidth}\centering
            \begin{tikzpicture}
                \node(H1) at (0,0) {$h$};
                \node(H2) at (1,0) {$h$} edge[double] (H1);
                \node(t1) at (-0.5,1) {$h$} edge[double] (H1);
                \node(t2) at (0.5,1) {$\partial$} edge (H2);
                \node(t3) at (1.5,1) {$\partial$} edge (H2);
                \node(t4) at (1.5,-1) {$h$} edge[double] (H2);
                \node(t5) at (0.5,-1) {$\partial$} edge (H1);
                \node(t6) at (-0.5,-1) {$\partial$} edge (H1);
            \end{tikzpicture}
        \end{subfigure}
        \begin{subfigure}{0.45\textwidth}\centering
            \begin{tikzpicture}\centering
                \node(H1) at (0,0) {$h$};
                \node(H2) at (1,0) {$h$} edge[double] (H1);
                \node(t1) at (-0.5,1) {$h$} edge[double] (H1);
                \node(t2) at (0.5,1) {$\partial$} edge (H1);
                \node(t3) at (1.5,1) {$\partial$} edge (H2);
                \node(t4) at (1.5,-1) {$h$} edge[double] (H2);
                \node(t5) at (0.5,-1) {$\partial$} edge (H2);
                \node(t6) at (-0.5,-1) {$\partial$} edge (H1);
                \end{tikzpicture}
            \end{subfigure}
            \caption{}\label{fig:swapboundarygraph}
        \end{subfigure}
        \begin{subfigure}{0.34\textwidth}\centering
            \begin{subfigure}{0.45\textwidth}\centering
                \begin{tikzpicture}
                \node(H1) at (0,0) {$h$};
                \node(H2) at (1,0) {$h$} edge[double] (H1);
                \node(t1) at (-0.5,1) {$h$} edge[double] (H1);
                \node(t2) at (0.5,1) {$\partial$} edge (H2);
                \node(t3) at (1.5,1) {$h$} edge[double] (H2);
                \node(t4) at (1.5,-1) {$\partial$} edge (H2);
                \node(t5) at (0.5,-1) {$\partial$} edge (H1);
                \node(t6) at (-0.5,-1) {$\partial$} edge (H1);
                \end{tikzpicture}
            \end{subfigure}
            \begin{subfigure}{0.45\textwidth}\centering
                \begin{tikzpicture}
                \node(H1) at (0,0) {$h$};
                \node(H2) at (1,0) {$h$} edge[double] (H1);
                \node(t1) at (-0.5,1) {$h$} edge[double] (H1);
                \node(t2) at (0.5,1) {$\partial$} edge (H1);
                \node(t3) at (1.5,1) {$h$} edge[double] (H2);
                \node(t4) at (1.5,-1) {$\partial$} edge (H2);
                \node(t5) at (0.5,-1) {$\partial$} edge (H2);
                \node(t6) at (-0.5,-1) {$\partial$} edge (H1);
            \end{tikzpicture}
        \end{subfigure}
        \caption{}\label{fig:swappairgraph}
    \end{subfigure}
    \caption{\eqsubref{fig:movehyperbolicgraph}--\eqsubref{fig:movehyperbolictile} Moving an $h$-tile onto the splitting cycle (marked by the double lines). \eqsubref{fig:movepartialgraph} Moving a $\partial$-tile closer to the splitting cycle (by sliding along the dashed path). \eqsubref{fig:swapboundarygraph} Changing two not-good $h$-tiles (with $\partial$-tiles on the same side of the splitting cycle) into two good $h$-tiles (with $\partial$-tiles across). \eqsubref{fig:swappairgraph} Moving a not-good $h$-tile along the splitting cycle.}\label{fig:fixannulus}
\end{figure}

\begin{proof}[Proof of Theorem~\ref{thm:Markov}]
    Suppose $\lambda_-$ and $\lambda_+$ are book links of the same book link type. By Lemma~\ref{lem:annuli} there is a collection book links $\lambda_-=\lambda_0, \cdots, \lambda_n=\lambda_+$ also of the same book link type interpolating between $\lambda_-$ and $\lambda_+$ together with collections of disjoint embedded annuli $A_i$ connecting $\lambda_{i-1}$ and $\lambda_i$. It suffices to treat the case of a single annulus $A$ bounded by book knots $\lambda_-$ and $\lambda_+$ and show how it generates a book link isotopy between suitably stabilized versions of the two book links. 
    
    We may assume by Lemma~\ref{lem:generalposition} that $A$ has an open book foliation. We may also assume it only has $\partial$- and $h$-tiles, possibly after applying a sequence of stabilizations to $\lambda_-$ and $\lambda_+$, as in Lemma~\ref{lem:onlytype0}. Since $A$ is an annulus, there is a separating cycle that passes through singular leaves and hyperbolic singularities and cuts $A$ into two regions; equivalently, the graph of singular leaves contains a single cycle. In what follows, we will alter the surface so that all $h$-tiles lie on this cycle and have one $\partial$-tile on each side of it, i.e., have one adjacent extremal singularity on each boundary component.
    
    First, we will alter the surface inductively so that all $h$-tiles are adjacent to both boundary components, equivalently, so that all hyperbolic singularities are on the separating cycle in the annulus and their corresponding vertices lie in the single cycle in the graph. Let us first consider a special case, if we can find an $h$-tile that is adjacent to but not on the cycle, i.e., which touches only one boundary component but is adjacent to an $h$-tile which touches both. If this tile has an adjacent $\partial$-tile, then we apply Lemma~\ref{lem:rotate}; see Figure~\ref{fig:movehyperbolicgraph}--\ref{fig:movehyperbolictile}. From the graphical point of view, if the cycle (the double edges in the figure) runs through one of the $h$-tiles but misses the other, then the perturbation will add the second $h$-tile to the cycle. (Note that there are two non-equivalent ways for the cycle to run through the left $h$-tile, but the perturbation works for both.) From the point of view of the foliated surface, if one $h$-tile touches only one boundary component but is adjacent to a tile which touches both, then the transformation shifts one of the components to be shared by both $h$-tiles, and both tiles touch both components.

    For the general case, select an $h$-tile $v$ adjacent to the cycle. If it does not have an attached $\partial$-tile, consider the tree branching of the cycle at $v$, traveling along its edges and through $h$-tile vertices, until a $\partial$-tile is reached. Then we may move the $\partial$-tile along this path using Lemma~\ref{lem:rotate}, as shown in Figure~\ref{fig:rotategraph}, where dashed lines are the path back toward $v$. When the $\partial$-tile reaches $v$, perform Lemma~\ref{lem:rotate} to move $v$ into the cycle.

    
    Next, once all $h$-tiles fall in a single cycle, each must be adjacent to exactly two other $h$-tiles as well as two $\partial$-tiles. We now ensure that each $h$-tile has its two $\partial$-tiles across the singularity from one another and so on different boundary components. We will call such a hyperbolic tile a \emph{good} tile. Recall that there are exactly $2d$ extremal singularities on each boundary component; for each not-good tile with a pair of $\partial$-tiles on $\lambda_-$, there is another not-good tile with its pair of $\partial$-tiles on $\lambda_+$. If these two tiles are adjacent, then we apply Lemma~\ref{lem:rotate} as in Figure~\ref{fig:swapboundarygraph} to move the $\partial$-tiles and ensure they fall opposite one another across the $h$-tiles. On the other hand, if there are two such not-good tiles, but they are separated by good tiles, then we may still apply Lemma~\ref{lem:rotate} to permute a not-good tile with a good tile until the pair of not-good tiles is adjacent, as in Figure~\ref{fig:swappairgraph}.
    
    Finally, we may assume that there is a choker of $2d$ of the $h$-tiles around the center of the annulus $A$, and each has two $\partial$-vertices attached on opposite sides. In the braid case, the annulus was foliated entirely by non-singular arcs connecting the two boundary components, provided a recipe for pushing $\lambda_-$ along the surface, each point following the appropriate non-singular arc to $\lambda_+$; this was an isotopy, and it did not introduce any critical points. Similarly, an $h$-tile with its pair of $\partial$-tiles on opposite boundary component foliation suggests a method for pushing $\lambda_-$ smoothly across the tile to $\lambda_+$, maintaining the critical point all along: see Figure~\ref{fig:makeisotopy_flat}-\ref{fig:makeisotopy_3d}.

    \begin{figure}\centering
        \begin{subfigure}[c]{0.3\textwidth}\centering
            {\pgfkeys{/pgf/fpu/.try=false}%
\ifx\XFigwidth\undefined\dimen1=0pt\else\dimen1\XFigwidth\fi
\divide\dimen1 by 1617
\ifx\XFigheight\undefined\dimen3=0pt\else\dimen3\XFigheight\fi
\divide\dimen3 by 1127
\ifdim\dimen1=0pt\ifdim\dimen3=0pt\dimen1=3946sp\dimen3\dimen1
  \else\dimen1\dimen3\fi\else\ifdim\dimen3=0pt\dimen3\dimen1\fi\fi
\tikzpicture[x=+\dimen1, y=+\dimen3]
{\ifx\XFigu\undefined\catcode`\@11
\def\temp{\alloc@1\dimen\dimendef\insc@unt}\temp\XFigu\catcode`\@12\fi}
\XFigu3946sp
\ifdim\XFigu<0pt\XFigu-\XFigu\fi
\clip(2595,125145) rectangle (4212,126272);
\tikzset{inner sep=+0pt, outer sep=+0pt}
\pgfsetfillcolor{black}
\pgftext[base,left,at=\pgfqpointxy{3908}{126145}] {\fontsize{8}{9.6}\usefont{T1}{ptm}{m}{n}$\lambda_0$}
\pgfsetlinewidth{+15\XFigu}
\pgfsetcolor{black}
\filldraw  (3248,125680) circle [radius=+55];
\filldraw  (3248,125215) circle [radius=+55];
\filldraw  (3248,125447) circle [radius=+55];
\filldraw  (3248,125912) circle [radius=+55];
\pgfsetlinewidth{+30\XFigu}
\draw (2628,125215)--(3869,125215);
\pgfsetlinewidth{+15\XFigu}
\draw (3869,126145)--(3869,125215);
\draw (2783,126145)--(3714,125215);
\draw (2628,126145)--(2628,125215);
\pgfsetlinewidth{+30\XFigu}
\draw (2628,126145)--(3869,126145);
\draw (2628,125680)--(3869,125680);
\draw (2628,125447)--(3869,125447);
\draw (2628,125912)--(3869,125912);
\pgfsetlinewidth{+7.5\XFigu}
\pgfsetstrokecolor{white}
\pgfsetdash{{+60\XFigu}{+60\XFigu}}{++0pt}
\draw (3947,125641)--(4200,125550);
\pgfsetlinewidth{+15\XFigu}
\pgfsetstrokecolor{black}
\pgfsetdash{}{+0pt}
\draw (2783,125215)--(3714,126145);
\pgfsetbeveljoin
\draw (2705,126145)--(2706,126142)--(2709,126137)--(2714,126126)--(2721,126111)--(2730,126090)
  --(2741,126065)--(2754,126036)--(2767,126006)--(2780,125974)--(2793,125942)
  --(2805,125912)--(2816,125882)--(2826,125854)--(2835,125828)--(2842,125804)
  --(2848,125781)--(2853,125760)--(2856,125739)--(2858,125719)--(2860,125699)
  --(2860,125680)--(2860,125661)--(2858,125641)--(2856,125621)--(2853,125600)
  --(2848,125579)--(2842,125556)--(2835,125532)--(2826,125506)--(2816,125478)
  --(2805,125448)--(2793,125418)--(2780,125386)--(2767,125354)--(2754,125324)
  --(2741,125295)--(2730,125270)--(2721,125249)--(2714,125234)--(2709,125223)
  --(2706,125218)--(2705,125215);
\draw (3791,126145)--(3790,126142)--(3787,126137)--(3782,126126)--(3775,126111)--(3766,126090)
  --(3755,126065)--(3742,126036)--(3729,126006)--(3716,125974)--(3703,125942)
  --(3691,125912)--(3680,125882)--(3670,125854)--(3662,125828)--(3655,125804)
  --(3649,125781)--(3644,125760)--(3641,125739)--(3638,125719)--(3637,125699)
  --(3636,125680)--(3637,125661)--(3638,125641)--(3641,125621)--(3644,125600)
  --(3649,125579)--(3655,125556)--(3662,125532)--(3670,125506)--(3680,125478)
  --(3691,125448)--(3703,125418)--(3716,125386)--(3729,125354)--(3742,125324)
  --(3755,125295)--(3766,125270)--(3775,125249)--(3782,125234)--(3787,125223)
  --(3790,125218)--(3791,125215);
\draw (2938,126145)--(2941,126143)--(2948,126138)--(2960,126130)--(2976,126119)--(2998,126105)
  --(3022,126089)--(3047,126073)--(3073,126057)--(3098,126042)--(3122,126029)
  --(3143,126018)--(3164,126009)--(3182,126002)--(3200,125996)--(3216,125992)
  --(3232,125990)--(3248,125990)--(3264,125990)--(3280,125992)--(3297,125996)
  --(3314,126002)--(3333,126009)--(3353,126018)--(3375,126029)--(3398,126042)
  --(3423,126057)--(3449,126073)--(3475,126089)--(3499,126105)--(3520,126119)
  --(3537,126130)--(3549,126138)--(3556,126143)--(3559,126145);
\draw (2938,125215)--(2941,125217)--(2948,125222)--(2960,125230)--(2976,125241)--(2998,125255)
  --(3022,125271)--(3047,125287)--(3073,125303)--(3098,125317)--(3122,125330)
  --(3143,125341)--(3164,125351)--(3182,125358)--(3200,125363)--(3216,125367)
  --(3232,125369)--(3248,125370)--(3264,125369)--(3280,125367)--(3297,125363)
  --(3314,125358)--(3333,125351)--(3353,125341)--(3375,125330)--(3398,125317)
  --(3423,125303)--(3449,125287)--(3475,125271)--(3499,125255)--(3520,125241)
  --(3537,125230)--(3549,125222)--(3556,125217)--(3559,125215);
\pgftext[base,left,at=\pgfqpointxy{3908}{125912}] {\fontsize{8}{9.6}\usefont{T1}{ptm}{m}{n}$\lambda_{0.25}$}
\pgftext[base,left,at=\pgfqpointxy{3908}{125215}] {\fontsize{8}{9.6}\usefont{T1}{ptm}{m}{n}$\lambda_1$}
\pgftext[base,left,at=\pgfqpointxy{3908}{125447}] {\fontsize{8}{9.6}\usefont{T1}{ptm}{m}{n}$\lambda_{0.75}$}
\pgftext[base,left,at=\pgfqpointxy{3908}{125680}] {\fontsize{8}{9.6}\usefont{T1}{ptm}{m}{n}$\lambda_{0.5}$}
\filldraw  (3248,126145) circle [radius=+55];
\endtikzpicture}%
\caption{}\label{fig:makeisotopy_flat}
        \end{subfigure}
        \begin{subfigure}[c]{0.3\textwidth}\centering
            \input{makeisotopy_3d.tikz}\caption{}\label{fig:makeisotopy_3d}
        \end{subfigure}
        \begin{subfigure}[c]{0.3\textwidth}\centering
            \includegraphics[scale=0.5]{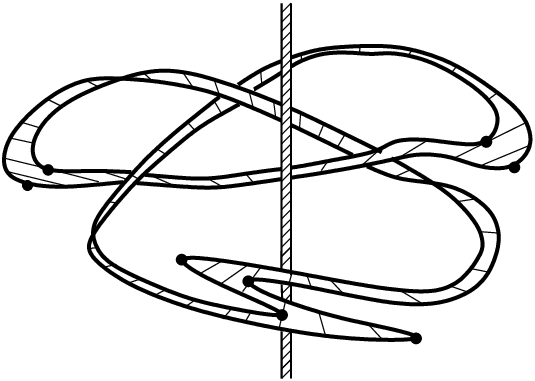}\caption{}\label{fig:generic}
        \end{subfigure}
        \caption{(\subref{fig:makeisotopy_flat}-\subref{fig:makeisotopy_3d}) An isotopy $\lambda_t$ generated by an $h$-tile with a pair of $\partial$-tiles. \eqsubref{fig:generic} A foliated annulus between book knots consisting only of tiles with extremal critical points.}\label{fig:makeisotopy}
    \end{figure}
\end{proof}

\begin{rem}
    The methods of the proof above may be applied to any surface bounded by a book link: it may be perturbed or stabilized to remove all tiles except $\partial$- and $h$-tiles, and these may then be rearranged, with some restrictions. Additionally, each of the $h$-tiles with an adjacent $\partial$-tile may be resolved in the reverse of Figure~\ref{fig:boundarysingularitychange}; in the case of the annulus arising in the proof of the theorem, this will remove all the $h$-tiles and result in a surface such as Figure~\ref{fig:generic}. For another example, see Figure~\ref{fig:seifert} for a Seifert surface for the trefoil with an open book foliation and resolved to contain only extremal singularities and their tiles.
\end{rem}

\begin{figure}[ht]
    \includegraphics[scale=0.5]{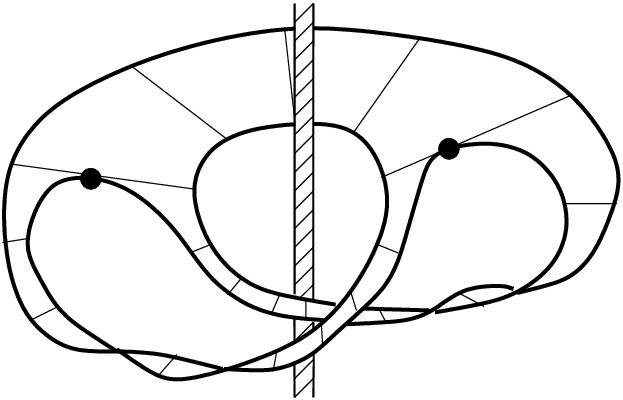}\caption{An open book foliation containing only extremal critical points.}\label{fig:seifert}
\end{figure}

\bibliographystyle{alpha}
\bibliography{bibliography}
\end{document}